\documentclass{article}
\usepackage{multicol,graphicx,color}
\usepackage{pslatex}
\usepackage{authblk}
\usepackage{amsthm}
\usepackage{amsmath}
\usepackage{amssymb}
\usepackage{latexsym}
\usepackage{lscape}
\usepackage{epsfig}
\usepackage{pstricks}
\usepackage{amsfonts}
\usepackage{mathrsfs}
\usepackage{mathrsfs}
\UseRawInputEncoding
\usepackage[
  hmarginratio={1:1},     % equal left and right margins
  vmarginratio={1:1},     % equal top and bottom margins=
  textwidth=15cm,        % new text width
  textheight=21cm,
  heightrounded,          % always useful
]{geometry}

\usepackage{graphicx,color}
\usepackage[colorlinks]{hyperref}
\hypersetup{linkcolor=blue,citecolor=blue,filecolor=black,urlcolor=blue}

\theoremstyle{plain}
\newtheorem{theorem}{Theorem}

\newtheorem{lemma}[theorem]{Lemma}

\theoremstyle{definition}
\newtheorem{definition}[theorem]{Definition}
\newtheorem{remark}[theorem]{Remark}

\numberwithin{equation}{section}
\numberwithin{theorem}{section}

\allowdisplaybreaks[4]%%%%%%%%%允许大公式断行

\author{Qian Gao, Qun Wang, Xiaojun Chang}
\title{Normalized ground state solutions of Schr\"odinger-KdV system in $\mathbb{R}^3$}
\date{}

\begin{document}

\maketitle

\begin{abstract}
\noindent In this paper, we study the coupled Schr\"odinger-KdV system
 \begin{align*}
\begin{cases}
	-\Delta u +\lambda_1 u=u^3+\beta uv~~&\text{in}~~\mathbb{R}^{3}, \\-\Delta v +\lambda_2 v=\frac{1}{2}v^2+\frac{1}{2}\beta u^2~~&\text{in}~~\mathbb{R}^{3}
\end{cases}
\end{align*}
subject to the mass constraints
\begin{equation*}
\int_{\mathbb{R}^{3}}|u|^2 dx=a,\quad \int_{\mathbb{R}^{3}}|v|^2 dx=b,
\end{equation*}
 where $a, b>0$ are given constants, $\beta>0$, and the frequencies $\lambda_1,\lambda_2$ arise as Lagrange multipliers. The system  exhibits
$L^2$-supercritical growth. Using a novel constraint minimization approach, we demonstrate the existence of a local minimum solution to the system. Furthermore, we establish the existence of normalized ground state solutions.
\end{abstract}

\medskip

{\small \noindent \text{Key Words:}  Normalized solutions; $L^2$-supercritical growth;  Schr\"odinger-KdV System; Variational methods.}

\medskip

\section{Introduction and main results}\label{intro}

We consider the following coupled Schr\"odinger-KdV system
\begin{align}
\begin{cases}
	if_t+f_{xx}+\beta fg+|f|^2f=0,& \\g_t+g_{xxx}+gg_{x}+\frac{1}{2}\beta (|f|^2)_x=0,&      \label{1.2+}
\end{cases}
\end{align}
where $f=f(x,t)$ is a complex-valued function, $g=g(x,t)$ is real-valued, $\beta$ is a real constant.
System (\ref{1.2+}) appears as models for interaction phenomena between long waves and short
waves, for example, resonant interaction between long and short capillary-gravity water waves, we refer the reader to \cite{fujia5,fujia6}, and the references therein.

When seeking the solitary wave solutions of (\ref{1.2+}) in the form
\begin{align*}
	f(x,t)=e^{i(\omega t+\kappa x)}u(x-ct),\quad g(x,t)=v(x-ct),
\end{align*}
and setting $\lambda_1=\kappa^2+\omega$, $\lambda_2=c=2\kappa$, the pair $(u,v)$ solves the following stationary Schr\"odinger-KdV system
\begin{align}\label{1.1}
\begin{cases}
	-\Delta u +\lambda_1 u=u^3+\beta uv~~&\text{in}~~\mathbb{R}^{N}, \\-\Delta v +\lambda_2 v=\frac{1}{2}v^2+\frac{1}{2}\beta u^2~~&\text{in}~~\mathbb{R}^{N}.
\end{cases}
\end{align}
Several authors have studied bound state and ground state solutions of (\ref{1.1}) when the frequencies $\lambda_1, \lambda_2$ are fixed, see, for example, \cite{A2003,B2018,E2015,C2017,E2023,L2013}.

This paper focuses on studying solutions to \eqref{1.1} with $N=3$ under the mass constraints
\begin{equation}\label{1.2}
\int_{\mathbb{R}^{3}}|u|^2 dx=a,\quad \int_{\mathbb{R}^{3}}|v|^2 dx=b,
\end{equation}
where $a, b>0$ are prescribed constants. It is standard that the solutions of \eqref{1.1}-\eqref{1.2} can be obtained as critical points of the energy functional
$$J(u,v)=\frac{1}{2}\int_{\mathbb{R}^{3}}|\nabla u|^2+|\nabla v|^2dx-\frac{1}{4}\int_{\mathbb{R}^{3}}u^4dx-\frac{1}{6}\int_{\mathbb{R}^{3}}v^3dx-\frac{1}{2}\beta\int_{\mathbb{R}^{3}}u^2vdx$$
on the constraint set $$S(a,b)=\left\{(u,v)\in H^1(\mathbb{R}^{3})\times H^1(\mathbb{R}^{3}):\int_{\mathbb{R}^{3}}|u|^2 dx=a,\quad\int_{\mathbb{R}^{3}}|v|^2 dx=b\right\}.$$
In this context, $\lambda_1, \lambda_2$ in (\ref{1.1}) appear as Lagrange multipliers and are prior unknown.

We recall the nonlinear Schr\"odinger equation
\begin{equation}\label{1.3}
	\begin{cases}
	-\Delta u+\lambda u=f(u),\quad  u\in H^{1}(\mathbb{R}^N),\\\int_{\mathbb{R}^N} |u|^2 dx=c>0.
	\end{cases}
\end{equation}
It is well known that the solutions of \eqref{1.3} depend on the asymptotic behavior of $f$ at infinity. Assume that $f(u)\sim|u|^{p-2}u$ as $|u|\rightarrow\infty$. Typically, when $2<p<2+\dfrac{4}{N}$, we say the equation satisfies $L^2$-subcritical growth; when $p=2+\dfrac{4}{N}$, the equation satisfies $L^2$-critical growth; when $2+\dfrac{4}{N}<p<2^*:=\dfrac{2N}{N-2}$, the equation satisfies $L^2$-supercritical growth. For studies on normalized solutions of (\ref{1.3}) with $L^2$-subcritical growth or $L^2$-critical growth, one may refer to \cite{LT1982,GS2014,H2019,M2014}. For the  $L^2$-supercritical growth case, see \cite{Alves2022,BdeV2013,B2017,B2021,CLY2023,IT2019,J1997,JJLV2022,JL2022,M2022,NS2020,S2020,W2022} and the references therein.

In 2016, Bartsch, Jeanjean and Soave \cite{BJS2016} studied the existence of normalized solutions for the nonlinear Schr\"odinger system
\begin{align*}	
\begin{cases}
	-\Delta u +\lambda_1 u=\mu_1u^3+\beta uv^2~~&\text{in}~~\mathbb{R}^{3}, \\-\Delta v+\lambda_2 v=\mu_2v^3+\beta u^2v~~&\text{in}~~\mathbb{R}^{3}
\end{cases}
\end{align*}
satisfying the constraint condition (\ref{1.2}), where $\mu_1,\mu_2,a,b$ are positive fixed quantities. They demonstrated that positive normalized solutions exist for different ranges of the coupling parameter $\beta>0$.
Bartsch and Jeanjean \cite{BJ2018} investigated the normalized solutions of a more general Schr\"odinger system
\begin{align}\label{1.5}	
\begin{cases}
	-\Delta u_1 =\lambda_1 u_1+\mu_1|u_1|^{p_1-2}u_1+\beta r_1 |u_1|^{r_1-2}|u_2|^{r_2}u_1~~&\text{in}~~\mathbb{R}^{N}, \\
-\Delta u_2=\lambda_2 u_2+\mu_2|u_2|^{p_2-2}u_2+\beta r_2 |u_1|^{r_1}|u_2|^{r_2-2}u_2~~&\text{in}~~\mathbb{R}^{N},
\end{cases}
\end{align}
under the same mass constraint condition, where $N\geq2,p_1,p_2\in(2,2^*)$, and $\beta,\mu_1,\mu_2,r_1,r_2,a,b>0$ with $2\leq r_1+r_2<2^*$. They addressed the case of $2<p_1<2+4/N<p_2,2+4/N<r_1+r_2<2^*,r_2>2$ in dimensions $2\leq N\leq 4$, and the scenarios where all of these values $p_1,p_2,r_1+r_2$ are either less than $2+4/N$ or all are greater than $2+4/N$, in the space $H^1_{rad}(\mathbb{R}^N)\times H^1_{rad}(\mathbb{R}^N)$. Subsequently, Gou and Jeanjean \cite{GJ2018}  extended these results to cover two new ranges of parameters. They considered the existence of multiple positive solutions to (\ref{1.5}) when $N\geq1,2<p_1,p_2<2+4/N,r_1,r_2>1,2+4/N<r_1+r_2<2^*$, and when $N\geq1,2+4/N<p_1,p_2<2^*,r_1,r_2>1,r_1+r_2<2+4/N$. In both scenarios, they obtained a local minimizer and a mountain pass solution, provided that $\beta>0$ was sufficiently small.
Recently, Jeanjean, Zhang and Zhong \cite{J2023} explicitly identified new ranges of $\beta>0$ for which there exists a solution to Schr\"odinger system (\ref{1.5}) in the case when $0<u,v\in H^1(\mathbb{R}^N), 1\leq N\leq4, 2+\frac{4}{N}<p,\ q,\ r_1+r_2<2^*$.
%\begin{align}	
%\begin{cases}
%-\Delta u +\lambda_1 u=\mu_1u^{p-1}+\beta r_1u^{r_1-1}v^{r_2}~~&\text{in}~~\mathbb{R}^{N}, \\
%-\Delta v+\lambda_2 v=\mu_2v^{q-1}+\beta r_2u^{r_1}v^{r_2-1}~~&\text{in}~~\mathbb{R}^{N},\\
%0<u,v\in H^1(\mathbb{R}^N),
%\int_{\mathbb{R}^{3}}u^2 dx=a, \quad\int_{\mathbb{R}^{3}}|v|^2 dx=b,
%\end{cases}
%\nonumber \end{align}
Moreover, for further discussions on normalized solutions for nonlinear Schr\"odinger systems, one may refer to \cite{BZ2021,G2016,G2019,K1895,LWT2022,N2019} and the references therein.

Since competing quadratic and cubic nonlinearities is a general physical phenomenon, it is important to understand the effect of such competition on normalized solutions. In 2022, Luo et al. \cite{LWY2022} provided a comprehensive study on the existence and non-existence of solutions to the following system
\begin{align}\label{1.6}
\begin{cases}
	-\Delta u +\lambda_1 u=\mu_1u^3+\rho  uv^2+\beta uv~~&\text{in}~~\mathbb{R}^{N}, \\-\Delta v+\lambda_2 v=\mu_2v^3+\rho u^2v+\frac{\beta}{2}u^2~~&\text{in}~~\mathbb{R}^{N}
\end{cases}
\end{align}
under the mass condition (\ref{1.2}), where $N\geq1,\mu_1,\mu_2,\rho,a,b>0$ and $\beta\in\mathbb{R}$. In the one-dimensional case, they found that normalized ground states exist and are obtained as global minimizers. For $N=2$, under suitable conditions on $a$ and $b$, they proved the existence
of normalized solutions. For $N=3$, they showed that at least two normalized solutions exist, one being a ground state and the other an excited state. Furthermore, they established the existence of a normalized ground state when $N=4$ and $\rho>0$, and they obtained non-existence results for $N\geq4$ when $\rho<0$.

In this paper, we are concerned with the normalized ground state solution for \eqref{1.1}, which is defined as follows.
\begin{definition}
We say that $(u_1,v_1)\in S(a,b)$ is a normalized ground state solution of \eqref{1.1} if it is a solution having minimal energy among all the solutions which belongs to $S(a,b)$, i.e.,
$$dJ|_{S(a,b)}(u_1,v_1)=0,\quad J(u_1,v_1)=\inf\left\{J(u,v):dJ|_{S(a,b)}(u,v)=0,(u,v)\in S(a,b)\right\} .$$	
\end{definition}

In 2010, Dias, Figueira and Oliveira \cite{D2010} investigated the solutions for the Schr\"odinger-KdV system
\begin{align}
\begin{cases}
	-\phi^{\prime\prime}+\lambda_1\phi=\phi^3+\beta\phi\psi,& \\-\psi^{\prime\prime}+\lambda_2\psi=\frac{1}{2}\psi^2+\frac{\beta}{2}\phi^2
\end{cases}
\nonumber\end{align}
restricted to $$\mathbb{X}_{\mu}=\{(u,v)\in H^1(\mathbb{R})\times H^1(\mathbb{R}):\int_{\mathbb{R}}(|u|^2+|v|^2)dx=\mu\}.$$ Since $N=1$, the system exhibits $L^2$-subcritical growth and the corresponding energy functional is bounded from below on $\mathbb{X}_{\mu}$. By employing a constraint minimization method, they established the existence of normalized ground state solutions.

Recently, Liang, Wu and Tang \cite{LWT2022} studied the existence of normalized ground state solutions for the system \eqref{1.1}-\eqref{1.2} in $\mathbb{R}^{N}$ for $N=1,2$. In their studies, system \eqref{1.1} displays $L^2$-subcritical for $N=1$ and $L^2$-critical growth for $N=2$. They explored properties of the function $m(a,b):=\inf\limits_{(u,v) \in S(a,b)}J(u,v)$ and proved that for any $\beta>0$, a solution  $(\lambda_1,\lambda_2,\tilde{u},\tilde{v})$ to system \eqref{1.1} exists with $\lambda_1,\lambda_2>0$ and $\tilde{u},\tilde{v}$ being radially symmetric when $a, b>0$ if $N=1$; and for $0<a<\bar{a}$, $b>0$ if $N=2$, where $\bar{a}$ depends only on $N$.

In the present paper, our aim is to seek normalized ground state solutions for the coupled Schrodinger-Kdv system in $\mathbb{R}^3$ within the constraint set $S(a,b)$. It is worth noting that
 for $N=3$, the system \eqref{1.1} exhibits $L^2$-supercritical growth, implying that $J$ is unbounded from below on $S(a,b)$. As a result, the
conventional constraint minimization arguments for $J|_{S(a,b)}$ used in \cite{LWT2022} do not ensure the existence of normalized ground state solutions. Inspired by the work of Jeanjean et al. \cite{JJLV2022}, we will develop a novel constrained minimization approach for $J$ and establish the existence of normalized ground state solutions for system \eqref{1.1}.

Throughout the paper, we assume that $N=3$.
%Define $B_{4\alpha}:=\left\{(u,v)\in W:|\nabla u|_2^2+|\nabla v|_2^2\leq 4\alpha\right\}$, where $\alpha$ is a constant dependent on $a, b$.
Now, let's state our main results.	
\begin{theorem}\label{thm: main ex}
For any $a>0$, there exists $\bar{b}>0$ such that, for all $b\in(0, \bar{b})$, system \eqref{1.1} possesses a normalized solution $(\tilde{u},\tilde{v})$ associated with some $\lambda_1>0$ and $\lambda_2>0$, which acts as a local minimizer of the constraint functional $J|_{S(a,b)}$. Moreover, system \eqref{1.1} has a normalized ground state solution $(u^*,v^*)$ for some $\lambda_1^*>0$ and $\lambda_2^*>0$, and $\tilde{u},\tilde{v},u^*,v^*$ are radially symmetric.
\end{theorem}

\begin{remark}
It should be emphasized that we are currently unable to confirm whether the local minimizer is the normalized ground state solution.  Consequently, the solutions $(\tilde{u}, \tilde{v})$ and $(u^*, v^*)$ that we obtained may either be the same or different.
\end{remark}
\begin{remark}
As shown in Lemma \ref{3.2.1}, to address the issue of lack of compactness, we employ the Schwartz symmetric decreasing rearrangement method in $H^1(\mathbb{R}^3)$. This approach allows us to achieve strong convergence of the minimizing sequence $\{(u_n,v_n)\}$ to a radially symmetric function $(\tilde{u},\tilde{v})$ in $L^p(\mathbb{R}^3)$.  Moreover, the problem can also be tackled in the radially symmetric space, given that $H^1_r(\mathbb{R}^3)\hookrightarrow L^p(\mathbb{R}^3)$  is compact for any $2<p<6$. Thus, the proof of Lemma \ref{3.2.1} will be simplified.
\end{remark}

\begin{remark}
Our conclusion corresponds to system (\ref{1.5}) when $N=3$ with $p_1=4, p_2=3$, $\mu_1=1, \mu_2=\frac{1}{2}$, $r_1=1, r_2=1$. This yields $2<p_2<2+4/N<p_1<2^*$, $r_1=1<r_2$ and $r_1+ r_2<2+4/N$. Alternatively, it corresponds to system (\ref{1.6}) when $N=3$, $\mu_1=1, \mu_2=\frac{1}{2}$ and $\rho=0$. These scenarios have not been previously studied in the existing literature.
\end{remark}

\begin{remark}
Notice that the functional $J$ not only has a locally minimal geometric structure on $S(a,b)$, but also exhibits a mountain pass geometry. It is therefore reasonable to investigate the existence of mountain pass type normalized solutions, which will be explored in our subsequent research.
\end{remark}
We build upon the approach of Jeanjean et al. \cite{JJLV2022} to prove Theorem \ref{thm: main ex}. However, their framework was designed for a single equation, while the Schr\"odinger-KdV system comprises multiple equations, complicating the definition of the auxiliary function $f(c,\rho)$ as used in \cite{JJLV2022}.
Additionally, we encounter the issue that the sum of two elements within a ball may not necessarily remain within the same ball. Thus, new techniques are necessary to investigate the continuity and subadditivity of  $m(a,b)$ as $a$ and $b$ vary. To tackle these obstacles, we consider a small ball with a radius dependent on
$b$ and a large circular ring with a diameter also dependent on
$b$ within $S(a,b)$. By comparing the values of $J$ on the ball and the ring, we obtain a locally minimal geometric structure of the constraint functional $J|_{S(a,b)}$. Consequently, we can define the function $m(a,b)$ on the outer circle of this ring, as detailed in Section \ref{local}. It is important to note that, in analyzing the subadditivity of $m(a,b)$ as done in \cite{LWT2022}, it is crucial to define $m(a,b)$ for $a\ge0$ and $b\ge0$. However, currently, $m(a,b)$ is only defined for $a\ge0$ and $0<b<\bar{b}$. To restore compactness, we incorporate information from the system itself and then establish the existence of local minimizers. Finally, using minimization arguments on the solutions set, we give the existence of normalized ground state solutions.

The paper is organized as follows.
In Section \ref{preliminary}, we introduce some preliminaries and lemmas that will be utilized throughout the paper. Section \ref{local} is dedicated to analyzing the subadditivity of the functional concerning the prescribed mass, followed by the demonstration of the existence of normalized solutions serving as local minimizers for the functional, thereby establishing Theorem \ref{thm: main ex}. In Section \ref{ground}, we establish the existence of normalized ground state solutions for the system and complete the proof of Theorem \ref{thm: main ex}.

\section{Preliminary results}\label{preliminary}\setcounter{equation}{0}

For any $1\leq r<+\infty$, we denote by $L^r(\mathbb{R}^3)$ the standard Lebesgue space equipped with the norm $\|u\|_r=\left(\int_{\mathbb{R}^3}|u(x)|^r dx\right)^{\frac{1}{r}}$. Denote by $H^1(\mathbb{R}^{3})$ the usual Hilbert space with the scalar product and norm defined as follows:
$$
			\langle u,v\rangle=\int_{\mathbb{R}^3}\big(\nabla u\cdot\nabla v+uv\big)dx,\quad \|u\|=\langle u,u\rangle^{\frac{1}{2}}.
			$$
Let $W:=H^1(\mathbb{R}^{3})\times H^1(\mathbb{R}^{3})$ be the product Sobolev space endowed with the norm $|(u,v)|=\left(|u|^2+|v|^2\right)^{\frac{1}{2}}$ for any $(u,v)\in W$.
Moreover, we denote by $W^{-1}$ the dual space of $W$. We now recall the following Gagliardo-Nirenberg inequality.

	\begin{lemma} \cite{Weinstein}(Gagliardo-Nirenberg inequality) Let $N\ge2$ and $2 \leq p<2^*$. Then there exists a constant $C_{N,p}>0$ depending on $N$ and $p$ such that
            \begin{equation*}\label{gn}
			\|u\|_p\leq C_{N,p}\left\|\nabla u\right\|_2^{\gamma_p} \|u\|_2^{1-\gamma_p},~~\forall u\in H^1(\mathbb{R}^N),
            \end{equation*}
            where $\gamma_{p}=N\left(\frac{1}{2}-\frac{1}{p}\right).$	
	\end{lemma}	
	%\begin{theorem}\label{lion}
	%\cite{Willem}(Lions's lemma) If $\left(u_n\right)$ is bounded in $H^1(\mathbb{R}^{N})$ and
	%		$$
	%		 \sup _{y \in \mathbb{R}^{N}} \int_{B(y,r)}\left|u_n\right|^p dx\rightarrow0\quad\mbox{as}~~ n\rightarrow\infty,
	%		$$
	%		then $u_n\rightarrow 0\ in\ L^q(\mathbb{R}^{N})\ \text{for}\ q \in\left(2, 2^*\right)$.
	%\end{theorem}
%\begin{theorem}\label{BL}
%		\cite{12}(Brézis-Lieb lemma) Let $\Omega$ be an open set of $\mathbb{R}^{N}$ and let $(u_n)\subset L^p(\Omega)$,\  $p\in[1,\infty)$. If\\$(a)\ (u_n)$ is bounded in %$L^p(\Omega)$,\\$(b)\ u_n\rightarrow u\ almost\ everywhere\ on\ \Omega,$ then$$\lim _{n\rightarrow\infty}\left(|u_n|_p^p-|u_n-u|_p^p\right)=|u|_p^p.$$
%	\end{theorem}

Next, we introduce the rearrangement results given by \cite{Shibata-2017} (see also \cite{h4}). Denote by $B_r(0)$ a ball in $\mathbb{R}^N$ with radius $r>0$ and centered at the origin. The notation $|\Omega|$ represents the $N$-dimensional Lebesgue measure of any subset $\Omega$ of $\mathbb{R}^N$. Let $u: \mathbb{R}^N\to \mathbb{R}$ be a Borel measurable function. We define $u$ as vanishing at infinity if $\left|\left\{x\in\mathbb{R}^{N}:|u(x)|>t\right\}\right|<\infty$ for any $t>0$. Now, consider two Borel measurable functions, $u$ and $v$, which both vanish at infinity in $\mathbb{R}^{N}$. For any $t>0$, we define $A^*(u,v;t):=\left\{x\in \mathbb{R}^{N}:|x|<r\right\}$, where $r>0$ is chosen such that
$$\left|B_r(0)\right|=\left|\left\{x\in \mathbb{R}^{N}:|u(x)|>t\right\}\right|+\left|\left\{x\in \mathbb{R}^{N}:|v(x)|>t\right\}\right|,$$and $\{u,v\}^*(x)$ by$$\{u,v\}^*(x):=\int_0^\infty\chi_{ A^*(u,v;t)}(x)dt,$$where $\chi_A(x)$ is a characteristic function of the set $A\subset\mathbb{R}^{N}$.
\begin{lemma}\label{2.0.10}
\cite[Appendix A1]{h4} (i)\ The function $\{u,v\}^*$ is radially symmetric, non-increasing and lower semi-continuous. Furthermore, for each $t>0$, it holds $\left\{x\in \mathbb{R}^{N}:|u(x)|>t\right\}=A^*(u,v;t).$
\\(ii)\ Let $\Phi:[0,\infty)\rightarrow [0,\infty)$ be nondecreasing, lower semi-continuous, continuous at 0 and $\Phi(0)=0$. Then, $\left\{\Phi(u),\Phi(v)\right\}^*\\=\Phi\left(\{u,v\}^*\right).$
\\(iii)\ $\|\{u,v\}^*\|_p^p=\|u\|_p^p+\|v\|_p^p$ for $1\leq p<\infty$.
\\(iv)\ If $u,v\in H^1(\mathbb{R}^{N})$, then $\{u,v\}^*\in H^1(\mathbb{R}^{N})$, $\left\|\nabla\{u,v\}^*\right\|_2^2\leq \|\nabla u\|_2^2+\|\nabla v\|_2^2$. In addition, if $u,v \in \left(H^1(\mathbb{R}^{N})\cap C^1(\mathbb{R}^{N})\right)\setminus \{0\}$ are radially symmetric, positive and non-increasing, then
$$\int_{\mathbb{R}^{N}}\left|\nabla\{u,v\}^*\right|^2dx<\int_{\mathbb{R}^{N}}|\nabla u|^2dx+\int_{\mathbb{R}^{N}}|\nabla v|^2dx.$$
\\(v)\ Let $u_1,u_2,v_1,v_2\geq 0$ be $Borel$ measurable functions which vanish at infinity. Then $$\int_{\mathbb{R}^{N}}(u_1u_2+v_1v_2)dx\leq\int_{\mathbb{R}^{N}}\{u_1,v_1\}^*\{u_2,v_2\}^*dx.$$
\end{lemma}

We also need the following Liouville type result.
\begin{lemma}
\cite[Appendix A2]{h4}\label{Liouville} Assume that $p\in (1,\frac{N}{N-2}\big]$ if $N\geq 3$ and $p\in (1, +\infty)$ if $N=1,2$. If $u\in L^p(\mathbb{R}^{N})$ is smooth nonnegative function satisfying $-\Delta u\geq 0$ in $\mathbb{R}^N$, then $u\equiv0$.
\end{lemma}

Using Lemma \ref{Liouville}, we can determine the sign of Lagrange multipliers.
\begin{lemma}\label{2.0.13}
If $(u,v)$ weakly solves system \eqref{1.1} and $u,v\ge0$, $u\not\equiv0$, then $\lambda_1>0$, $\lambda_2>0.$
\end{lemma}
\begin{proof}
By the regularity theory, $u,v\in C^2(\mathbb{R}^3)$. Since $u\not\equiv0$, it is easily seen that $v\not\equiv 0$.
If $\lambda_1<0$, we can see that $$-\Delta u=u^3+\beta uv-\lambda_1u\geq 0.$$  Then, in view of $u\in H^1(\mathbb{R}^3)$, we conclude by Lemma \ref{Liouville} that $u\equiv0$, which is a contradiction. So $\lambda_1>0$. Similarly, we have $\lambda_2>0.$
\end{proof}

\section{Proof of Theorem \ref{thm: main ex}}\label{ground}\setcounter{equation}{0}

In this section, we will initially utilize a constraint minimization approach to demonstrate the existence of local minimum normalized solutions to equation \eqref{1.1}. Then, we will proceed to minimize over the solution set for normalized solutions to establish the existence of normalized ground state solutions.

\subsection{Local minimizer}\label{local}\setcounter{equation}{0}
For $k>0$, define $$
A_{1,k}:=\left\{(u,v)\in S(a,b): |\nabla u|_2^2+|\nabla v|_2^2\leq k\right\}, $$
$$A_{2,k}:=\left\{(u,v)\in S(a,b): 2k\leq|\nabla u|_2^2+|\nabla v|_2^2\leq 6k\right\}.
$$
\begin{lemma}\label{Local}
For any $a\ge0$, there exists $\bar{b}>0$ such that for any $0<b<\bar{b}$,
 $$\sup_{(u,v)\in A_{1,b}}J(u,v)<\inf_{(u,v)\in A_{2,b}}J(u,v).$$
\end{lemma}
	\begin{proof}
%For any $a$, $b\geq0$ and $(u,v)\in S(a,b)\cap B_{4~k}$, by the H\"older inequality and Lemma \ref{gn},
%\begin{equation*}\begin{aligned}
   %J(u,v)%=&\frac{1}{2}\int_{\mathbb{R}^{3}}|\nabla u|^2+|\nabla v|^2dx-\frac{1}{4}\int_{\mathbb{R}^{3}}u^4dx-\frac{1}{6}\int_{\mathbb{R}^{3}}v^3dx-\frac{1}{2}\beta\int_{\mathbb{R}^{3}}u^2vdx\\
          %\geq&\frac{1}{2}\left(\|\nabla u\|_{2}^2+\|\nabla v\|_{2}^2\right)-\frac{1}{4}\|u\|_4^4-\frac{1}{6}\|v\|_{3}^3-\frac{1}{2}\beta\|u\|_{4}^2\|v\|_{2}\\
		  %\geq&\frac{1}{2}\left(|\nabla u|_{2}^2+|\nabla v|_{2}^2\right)-\frac{1}{4}C_{3,4}^4|\nabla u|_2^3|u|_2-\frac{1}{6}C_{3,3}^3|\nabla v|_{2}^{\frac{3}{2}}|v|_{2}^{\frac{3}{2}}-\frac{1}{2}\beta C_{3,4}^2|\nabla u|_{2}^{\frac{3}{2}}|u|_{2}^{\frac{1}{2}}|v|_{2}\\
          %=&\frac{1}{2}\left(\|\nabla u\|_{2}^2+\|\nabla v\|_{2}^2\right)-\frac{1}{4}C_{3,4}^4a^{\frac{1}{2}}\|\nabla u\|_2^3-\frac{1}{6}C_{3,3}^3b^{\frac{3}{4}}\|\nabla v\|_{2}^{\frac{3}{2}}-\frac{1}{2}\beta C_{3,4}^2a^{\frac{1}{4}}b^{\frac{1}{2}}\|\nabla u\|_{2}^{\frac{3}{2}}.\label{3.1}
%\end{aligned}
%	\end{equation*}
For any $a\ge 0$ and $b>0$, let $(u_1,v_1)\in A_{1,k}$, $(u_2,v_2)\in A_{2,k}$ with $k>0$. By the H\"older inequality and Lemma \ref{gn} it follows that
\begin{equation}\begin{aligned}
		&J(u_2,v_2)-J(u_1,v_1) \nonumber\\
=&\frac{1}{2}\int_{\mathbb{R}^{3}}(|\nabla u_2|^2+|\nabla v_2|^2)dx-\frac{1}{4}\int_{\mathbb{R}^{3}}u_2^4dx-\frac{1}{6}\int_{\mathbb{R}^{3}}v_2^3dx-\frac{1}{2}\beta\int_{\mathbb{R}^{3}}u_2^2v_2dx\nonumber\\
		-&\frac{1}{2}\int_{\mathbb{R}^{3}}(|\nabla u_1|^2+|\nabla v_1|^2)dx+\frac{1}{4}\int_{\mathbb{R}^{3}}u_1^4dx+\frac{1}{6}\int_{\mathbb{R}^{3}}v_1^3dx+\frac{1}{2}\beta\int_{\mathbb{R}^{3}}u_1^2v_1dx\nonumber\\
          \geq&\frac{1}{2}k-\frac{1}{4}C_{3,4}^4a^{\frac{1}{2}}(6k)^{\frac{3}{2}}-\frac{1}{6}C_{3,3}^3b^{\frac{3}{4}}(6k)^{\frac{3}{4}}-\frac{1}{2}\beta C_{3,4}^2a^{\frac{1}{4}}b^{\frac{1}{2}}(6k)^{\frac{3}{4}}:=g(b,k).
		\end{aligned}\end{equation}
Consequently, we obtain
$$g(b,b):=\frac{1}{2}b-\left(\frac{1}{4}6^{\frac{3}{2}} C_{3,4}^4a^{\frac{1}{2}} +6^{-\frac{1}{4}}C_{3,3}^3\right)b^{\frac{3}{2}}-\frac{1}{2}6^{\frac{3}{4}}\beta C_{3,4}^2a^{\frac{1}{4}}b^{\frac{5}{4}}.$$
Upon direct computation, it is clear that for any $a\geq0$, there exists $\bar{b}=\bar{b}(a)>0$ such that for any $0<b<\bar{b}$, the inequality
 $J(u_2,v_2)>J(u_1,v_1)$ holds.
\end{proof}

Define $B_{k}:=\left\{(u,v)\in W: |\nabla u|_2^2+|\nabla v|_2^2\leq k\right\}$ for $k>0$. We then consider
the following minimization problem
$$m(a,b):=\inf\limits_{(u,v)\in S(a,b)\cap B_{6b}}J(u,v),\quad\forall a\ge0, b>0.
$$
The main result of this subsection is the following compactness lemma:
\begin{lemma}\label{3.2.1}
 Let $a\ge0, 0<b<\bar{b}$. Suppose $\{(u_n,v_n)\}\subset B_{6b}$ is a sequence such that $$J(u_n,v_n)\rightarrow m(a,b),\quad |u_n|_2^2\rightarrow a,\quad |v_n|_2^2\rightarrow b.$$
 Then there exists a sequence $\{y_n\}\subset \mathbb{R}^{3}$ and a pair $(u,v)\in S(a,b)$ such that $\left(u_n(x+y_n),v_n(x+y_n)\right)\rightarrow (u,v)$ strongly in $W$.
	\end{lemma}

Based on the above lemma, we can conclude that there exists a local minimizer $(u_1,v_1)$ of $J$ on $S(a,b)\cap B_{6b}$, thereby completing the proof of the first part of Theorem \ref{thm: main ex}. Furthermore, Lemma \ref{2.0.13} indicates that $\lambda_1,\lambda_2>0.$ To proceed with proving Lemma \ref{3.2.1}, let us first gather some properties of $m(a,b)$.

\begin{lemma}\label{3.1.2}
		For any $a\ge0, b>0$, we have $m(a,b)<0$.
	\end{lemma}
	\begin{proof}
		For any $(u,v)\in S(a,b)\cap B_{6b}$, take $u_t(x):=t^{\frac{3}{2}}u(tx),\ v_t(x):=t^{\frac{3}{2}}v(tx).$ Clearly, $(u_t(x),v_t(x))\in S(a,b)$. Moreover,
    there exists $t_0\in(0,1)$ small enough such that $(u_t(x),v_t(x))\in S(a,b)\cap B_{6b}$ for all $t\in(0, t_0)$.
Since
 \begin{equation}\begin{aligned}
   J(u_t,v_t)=&\frac{1}{2}\int_{\mathbb{R}^{3}}(|\nabla u_t|^2+|\nabla v_t|^2)dx-\frac{1}{4}\int_{\mathbb{R}^{3}}u_t^4dx-\frac{1}{6}\int_{\mathbb{R}^{3}}v_t^3dx-\frac{1}{2}\beta\int_{\mathbb{R}^{3}}u_t^2v_tdx\nonumber\\
             =&\frac{t^2}{2}\int_{\mathbb{R}^{3}}(|\nabla u|^2+|\nabla v|^2)dx-\frac{t^{3}}{4}\int_{\mathbb{R}^{3}}u^4dx-\frac{t^{\frac{3}{2}}}{6}\int_{\mathbb{R}^{3}}v^3dx-\frac{t^{\frac{3}{2}}}{2}\beta\int_{\mathbb{R}^{3}}u^2vdx\nonumber\\
             \rightarrow&0^-~~~~\mbox{as}~~t\rightarrow0^+.\nonumber
 \end{aligned}
	\end{equation}
It follows that the conclusion holds.
	\end{proof}

\begin{lemma}\label{3.1.3}
		$m(a,b)$ is continuous with respect to $a\ge0$ and $b\in (0, \bar{b})$.
	\end{lemma}
	\begin{proof}
For $a\ge0$ and $b\in (0, \bar{b})$, taking $a_n\ge0$ and $b_n\in (0, \bar{b})$ such that $\lim\limits_{n\rightarrow\infty}(a_n,b_n)=(a,b).$ By the definition of $m(a_n,b_n)$ and Lemma \ref{Local}, for any $\varepsilon>0$, there exists $(u_n,v_n)\in S(a_n,b_n)\cap B_{2b}$ such that
        \begin{equation}
		J(u_n,v_n)\leq m(a_n,b_n)+\varepsilon.\label{3.2}
		\end{equation}
Setting $w_n=\sqrt{\frac{a}{a_n}}u_n$, $z_n=\sqrt{\frac{b}{b_n}}v_n$, it follows that $(w_n,z_n)\in S(a,b).$  We denote $\epsilon_{1,n}=a-a_n$ and $\epsilon_{2,n}=b-b_n$. Obviously, $|\epsilon_{1,n}|\to0, |\epsilon_{2,n}|\to 0$ as $n\to+\infty$. By direct calculation, for sufficiently large $n$,
\begin{equation}\begin{aligned}
   |\nabla w_n|_2^2+|\nabla z_n|_2^2=&\frac{a}{a_n}\int_{\mathbb{R}^{3}}|\nabla u_n|^2+\frac{b}{b_n}\int_{\mathbb{R}^{3}}|\nabla v_n|^2dx\nonumber\\
   =&\left(1+\frac{\epsilon_{1,n}}{a-\epsilon_{1,n}}\right)|\nabla u_n|_2^2+\left(1+\frac{\epsilon_{2,n}}{b-\epsilon_{2,n}}\right)|\nabla v_n|_2^2\nonumber\\
             <&6b\nonumber,
 \end{aligned}
	\end{equation}
 which implies that $(w_n,z_n)\in B_{6b}$. Hence, by the boundedness of $\left\{(u_n,v_n)\right\}$, we obtain
$$
		\begin{aligned}
		m(a,b)\leq& J(w_n,z_n)\nonumber\\
              %=&\frac{1}{2}\frac{a}{a_n}\int_{\mathbb{R}^{3}}|\nabla u_n|^2dx+\frac{1}{2}\int_{\mathbb{R}^{3}}\frac{b}{b_n}|\nabla v_n|^2dx
               %-\frac{1}{4}\int_{\mathbb{R}^{3}}\frac{a^2}{a_n^2}u_n^4dx\nonumber\\
               %&-\frac{1}{6}\int_{\mathbb{R}^{3}}\frac{\sqrt{b^3}}{\sqrt{b_n^3}}v_n^3dx-\frac{1}{2}\beta\int_{\mathbb{R}^{3}}\frac{a\sqrt{b}}{a_n\sqrt{b_n}}u_n^2v_ndx\nonumber\\
               =&J(u_n,v_n)+\frac{1}{2}\left(\frac{a}{a_n}-1\right)\int_{\mathbb{R}^{3}}|\nabla u_n|^2dx+\frac{1}{2}\left(\frac{b}{b_n}-1\right)\int_{\mathbb{R}^{3}}|\nabla v_n|^2dx\nonumber\\
               -&\frac{1}{4}\left(\frac{a^2}{a_n^2}-1\right)\int_{\mathbb{R}^{3}}u_n^4dx-\frac{1}{6}\left(\frac{\sqrt{b^3}}{\sqrt{b_n^3}}-1\right)\int_{\mathbb{R}^{3}}v_n^3dx\nonumber\\
               -&\frac{1}{2}\beta\left(\frac{a\sqrt{b}}{a_n\sqrt{b_n}}-1\right)\int_{\mathbb{R}^{3}}u_n^2v_ndx\nonumber\\
               =&J(u_n,v_n)+o(1).
		\end{aligned}$$
In view of \eqref{3.2} we deduce that $$m(a,b)\leq m(a_n,b_n)+\varepsilon+o(1).$$

On the other hand, let $\{(w_n,z_n)\}$ be a minimizing sequence of $J$ at $m(a,b)$. Then, for any $\varepsilon>0$, there exists $(w_n,z_n)\in S(a,b)\cap B_{2b}$ such that $$J(w_n,z_n)\leq m(a,b)+\varepsilon.$$
We set $\tilde {w}_n=\sqrt{\frac{a_n}{a}}w_n, \tilde {z}_n=\sqrt{\frac{b_n}{b}}z_n.$ In a similar manner as above, $(\tilde {w}_n,\tilde {z}_n)\in S(a_n,b_n)\cap B_{6b}$. Moreover, we have
$$
m(a_n,b_n)\leq J(\tilde {w}_n,\tilde {z}_n)\le J(w_n,z_n)+o(1)\leq m(a,b)+\varepsilon+o(1).
$$
Since $\varepsilon$ is arbitrary, we deduce that $\lim\limits_{n\rightarrow\infty}m(a_n,b_n)=m(a,b).$
\end{proof}

\begin{lemma}\label{3.1.4} If $0\leq c\leq a, 0<d< b<\bar{b}$, then $m(a,b) \leq m(c,d)+m(a-c,b-d)$.
\end{lemma}
\begin{proof}
For any $\varepsilon>0$, there exist $(u_1,v_1)\in S(c,d)\cap B_{2b}$, $(u_2,v_2)\in S(a-c,b-d)\cap B_{2b}$ such that
\begin{equation*}\label{3.6}
J(u_1,v_1)\leq m(c,d)+\frac{\varepsilon}{2},\quad J(u_2,v_2)\leq m(a-c,b-d)+\frac{\varepsilon}{2}.
\end{equation*}
Without losing generality, we assume $u_1,u_2,v_1,v_2\geq0$. Let $w_1:=\{u_1,u_2\}^*,w_2:=\{v_1,v_2\}^*$. By Lemma \ref{2.0.10} (iii), we get
$$|w_1|_2^2=|u_1|_2^2+|u_2|_2^2=a,\quad |w_2|_2^2=|v_1|_2^2+|v_2|_2^2=b,$$
$$|\nabla w_1|_2^2+|\nabla w_2|_2^2<|\nabla u_1|_2^2+|\nabla u_2|_2^2+|\nabla v_1|_2^2+|\nabla v_2|_2^2<6b.$$
Therefore, $(w_1,w_2)\in S(a,b)\cap B_{6b}.$ From (ii)-(v) in Lemma \ref{2.0.10} and the definition of $m(a,b)$, it follows that
$$
		\begin{aligned}
		m(a,b)\leq& J(w_1,w_2)\nonumber\\
              =&\frac{1}{2}\int_{\mathbb{R}^{3}}(|\nabla w_1|^2+|\nabla w_2|^2)dx-\frac{1}{4}\int_{\mathbb{R}^{3}}w_1^4dx\nonumber \\
		       -&\frac{1}{6}\int_{\mathbb{R}^{3}}w_2^3dx-\frac{1}{2}\beta\int_{\mathbb{R}^{3}}w_1^2w_2dx\nonumber\\
             \leq& \frac{1}{2}\int_{\mathbb{R}^{3}}(|\nabla u_1|^2+|\nabla u_2|^2+|\nabla v_1|^2+|\nabla v_2|^2)dx\nonumber-\frac{1}{4}\int_{\mathbb{R}^{3}}(u_1^4+u_2^4)dx\\
               -&\frac{1}{6}\int_{\mathbb{R}^{3}}(v_1^3+v_2^3)dx-\frac{1}{2}\beta\int_{\mathbb{R}^{3}}(u_1^2v_1+u_2^2v_2)dx\nonumber\\
               =&J(u_1,v_1)+J(u_2,v_2)\nonumber\\
               \leq&m(c,d)+m(a-c,b-d)+\varepsilon.
		\end{aligned}
		$$
Since $\varepsilon$ is arbitrary, we get $m(a,b) \leq m(c,d)+m(a-c,b-d).$
\end{proof}
\begin{lemma}\label{3.1.5}
Let $\left\{(u_n,v_n)\right\}\subset S(a,b)\cap B_{6b}$ be a bounded minimizing sequence for $m(a,b)$. Then there exist $\alpha>0$ and a sequence $\left\{y_n\right\} \subset \mathbb{R}^3$ such that
\begin{align}\label{3.7}
\limsup_{n\rightarrow\infty}\int_{B_2(y_n)}|(u_n,v_n)|^2dx\geq \alpha>0.
\end{align}
	\end{lemma}
\begin{proof}
We assume by contradiction that \eqref{3.7} does not hold. Then by the Lions lemma (see \cite{W1996}) we get $ (u_n,v_n)\rightarrow(0,0)$ in $ L^p(\mathbb{R}^{3})\times L^p(\mathbb{R}^{3})$ for all $p\in(2,6)$, which implies that
$$\liminf\limits_{n\rightarrow\infty}J(u_n,v_n)=\liminf\limits_{n\rightarrow\infty}\int_{\mathbb{R}^{3}}(|\nabla u_n|^2+|\nabla v_n|^2)dx\geq0.$$
This contradicts the fact that $m(a,b)<0$.
\end{proof}

By similar arguments as in (see \cite{W1996}), we get the following Brezis-Lieb type result.
\begin{lemma}\label{3.1.6} If $(u_n,v_n)\rightharpoonup (u,v)\in W$, then
	\begin{eqnarray*}\label{4.24.1}
			\int_{\mathbb{R}^{3}}[u_n^2v_n-(u_n-u)^2(v_n-v)]dx=\int_{\mathbb{R}^{3}}u^2vdx+o(1).
\end{eqnarray*}
	\end{lemma}
\begin{proof}[Proof of Lemma \ref{3.2.1}]
By the boundedness of $\{(u_n,v_n)\}$ and Lemma $\ref{3.1.5}$, there exists a sequence $\{y_n\}\subset \mathbb{R}^{3}$ and $(u,v)\in W\setminus\{(0,0)\}$such that
\begin{itemize}
\item $\left(u_n(x+y_n),v_n(x+y_n)\right) \rightharpoonup (u,v)\ \ \text{in}\ W$;

\item $\left(u_n(x+y_n),v_n(x+y_n)\right)\rightarrow (u,v)\ \text{in}\ L_{loc}^{p}(\mathbb{R}^{3})\times L_{loc}^{p}(\mathbb{R}^{3})\ \text{for\ all}\ p\in[2,6)$;

\item $u_n(x+y_n)\rightarrow u(x), v_n(x+y_n)\rightarrow v(x)\ a.e.\ \ x\in\mathbb{R}^{3}.$
\end{itemize}
It is easily seen that $v\not\equiv 0$.
Let $\phi_n(x)=u_n(x+y_n)-u(x)$ and $\psi_n(x)=v_n(x+y_n)-v(x).$ We shall prove that $\left(\phi_n(x),\psi_n(x)\right)\rightarrow(0,0)\ \text{in}\ W.$ Using the Brezis-Lieb lemma, we get
\begin{equation}
		|\phi_n|_2^2=|u_n|_2^2-|u|_2^2+o(1)=a-|u|_2^2+o(1),\label{2.18}
		\end{equation}
		\begin{equation}
|\psi_n|_2^2=|v_n|_2^2-|v|_2^2+o(1)=b-|v|_2^2+o(1).\label{2.19}
\end{equation}
From weak convergence, it follows that
		\begin{equation}
		|\nabla\phi_n|_2^2=|\nabla u_n|_2^2-|\nabla u|_2^2+o(1),  \label{2.20}
		\end{equation}
\begin{equation}
|\nabla\psi_n|_2^2=|\nabla v_n|_2^2-|\nabla v|_2^2+o(1).\label{2.21}
\end{equation}
Furthermore, combining these and Lemma \ref{3.1.6}, there holds
		\begin{equation}
		J(u_n,v_n)=J(\phi_n,\psi_n)+J(u,v)+o(1) .\label{3-10}
		\end{equation}

We claim that $|\phi_n|_2^2 \rightarrow 0$, $|\psi_n|_2^2 \rightarrow 0.$\ It suffices to show that $|u|_2^2=a$ and $|v|_2^2=b$.\ By \eqref{2.18}-\eqref{2.21},\  for $n$ large enough, we have $$|\phi_n|_2^2 \leq a ,\quad |\psi_n|_2^2 \leq b,\quad|\nabla \phi_n|_2^2+|\nabla \psi_n|_2^2 \leq|\nabla u_n|_2^2+|\nabla v_n|_2^2\leq6b.$$ Hence, $\left(\phi_n,\psi_n\right)\in S\left(|\phi_n|_2^2,|\psi_n|_2^2\right)\cap B_{6b}$ and $J\left(\phi_n,\psi_n\right)\geq m\left(|\phi_n|_2^2,|\psi_n|_2^2\right)$.\ Recalling that $J(u_n,v_n)\rightarrow m(a,b)$, in view of \eqref{3-10}, we deduce that
\begin{equation*}
		m(a,b) \geq m\left(|\phi_n|_2^2,|\psi_n|_2^2\right)+J(u,v)+o(1) .\label{3.10}
		\end{equation*}
Since $m(a,b)$ is continuous, using \eqref{2.18}-\eqref{2.19}, we infer
$m(|\phi_n|^2_2,|\psi_n|^2_2)\rightarrow m\left(a-|u|_2^2,b-|v|_2^2\right)$.\ Then
		\begin{equation}
		m(a,b) \geq m\left(a-|u|_2^2,b-|v|_2^2\right)+J(u,v) .\label{219}
		\end{equation}
By the weak limit, we also have that $(u,v)\in S\left(|u|_2^2,|v|_2^2\right)\cap B_{6b}$. This implies that $J(u,v) \geq m(|u|_2^2,|v|_2^2)$. If $J(u,v)>m(|u|_2^2,|v|_2^2)$, it then follows from \eqref{219} and Lemma \ref{3.1.4} that
	\begin{equation*}
m(a,b)> m\left(a-|u|_2^2+|u|_2^2,b-|v|_2^2+|v|_2^2\right)=m(a,b),
	\end{equation*}
		which is impossible. Thus, we obtain $J(u,v)=m(|u|_2^2,|v|_2^2)$, and hence
\begin{align}\label{3-15}
m(a,b)=m\left(a-|u|_2^2,b-|v|_2^2\right)+m\left(|u|_2^2,|v|_2^2\right).
\end{align}

We divide three cases to prove $|u|_2^2=a$, $|v|_2^2=b$.

Case 1. If $0\leq |u|_2^2<a, 0<|v|_2^2<b,$ we suppose that $\{(w_n,z_n)\}$ is a minimizing sequence of $J$ at $ m\left(a-|u|_2^2,b-|v|_2^2\right)$. Clearly, $\{(w_n,z_n)\}$ is bounded in $W$. Then there exists a sequence $\{y_n\} \subset \mathbb{R}^{3}$ and $(w,z)\in W\setminus\{(0,0)\}$ such that
\begin{itemize}
\item $\left(w_n(x+y_n),~ z_n(x+y_n)\right) \rightharpoonup (w,z)\ \text{in}\ W$;

\item $\left(w_n(x+y_n),~ z_n(x+y_n)\right)\rightarrow (w,z)\ \text{in}\ L_{loc}^{p}(\mathbb{R}^{3})\times L_{loc}^{p}(\mathbb{R}^{3})\ \text{for\ all}\ p\in[2,6)$;

\item $w_n(x+y_n)\rightarrow w(x),~z_n(x+y_n)\rightarrow z(x)\ a.e.\ \ x\in\mathbb{R}^{3}.$
\end{itemize}
Moreover, by similar arguments as above we get
\begin{align}\label{3-17}
		          m\left(a-|u|_2^2,b-|v|_2^2\right)=m\left(a-|u|_2^2-|w|_2^2,b-|v|_2^2-|z|_2^2\right)+m\left(|w|_2^2,|z|_2^2\right),
		        \end{align}
and $J(w,z)=m\left(|w|_2^2,|z|_2^2\right)$. Let $\tilde{u},\tilde{v},\tilde{w},\tilde{z}$ be the classical Schwartz symmetric decreasing rearrangement of $|u|,|v|,|w|,|z|$. According to \cite{duichen}, $\tilde{u}$ has the following properties\\$(i)$ $\tilde{u}$ is nonegative;\\$(ii) $
	$	\int_{\mathbb{R}^{3}}\left|\nabla \tilde{u}\right|^2 d \xi \leq \int_{\mathbb{R}^{3}}\left|\nabla(\left|u\right|)\right|^2 d \xi \leq \int_{\mathbb{R}^{3}}\left|\nabla u\right|^2 d\xi;$\\$(iii)$
		$\int_{\mathbb{R}^{3}}\left|\tilde{u}\right|^t d \xi=\int_{\mathbb{R}^{3}}\left|u\right|^t d \xi,\  t \in[1, \infty);$  \\$(iv)$ $\int_{\mathbb{R}^{3}}\tilde{u}^2\tilde{v}dx\geq\int_{\mathbb{R}^{3}}u^2vdx.$
		\\Firstly, we can see that $\tilde{v},\tilde{z}\neq0$ since $v, z\neq0$. Moreover, by (ii)-(iv), one has $J(\tilde{u},\tilde{v})\le J(u,v)$.
We can deduce from $J(u,v)= m(|u|_2^2,|v|_2^2)$ and $m(|u|_2^2,|v|_2^2)=m(|\tilde{u}|_2^2,|\tilde{v}|_2^2)$ that		
\begin{align}\label{3-18}
		         J(\tilde{u},\tilde{v})=m(|u|_2^2,|v|_2^2).
		        \end{align}
Similarly, we obtain
\begin{align}\label{3-19}
		         J(\tilde{w},\tilde{z})=m(|w|_2^2,|z|_2^2).
		        \end{align}
Therefore, $\left(\tilde{u},\tilde{v}\right),\left(\tilde{w},\tilde{z}\right)$ are both local minimizers of $J$ on $S\left(|w|_2^2,|z|_2^2\right)\cap B_{6b}$. According to the elliptic regularity theory and strong maximum principle, it follows that $\tilde{u},\tilde{v},\tilde{w},\tilde{z}\in C^2(\mathbb{R}^{3})$ and $\tilde{v},\ \tilde{z}>0$. Hence, by Lemma \ref{2.0.10} (ii)-(v) we get $$\int_{\mathbb{R}^{3}}\left|\nabla\{\tilde{v},\tilde{z}\}^*\right|^2dx<\int_{\mathbb{R}^{3}}(|\nabla \tilde{v}|^2+|\nabla \tilde{z}|^2)dx.$$
Similarly, by Lemma \ref{2.0.10} (ii)-(v), we have
$$\int_{\mathbb{R}^{3}}\left|\nabla\{\tilde{u},\tilde{w}\}^*\right|^2dx\leq\int_{\mathbb{R}^{3}}(|\nabla \tilde{u}|^2+|\nabla \tilde{w}|^2)dx.$$
Moreover,
$$
\int_{\mathbb{R}^{3}}\left(\{\tilde{u},\tilde{w}\}^*\right)^2\{\tilde{v},\tilde{z}\}^*dx=\int_{\mathbb{R}^{3}}\left\{ \tilde{u}^2,\tilde{w}^2\right\}^*\left\{ \tilde{v},\tilde{z}\right\}^*dx\geq\int_{\mathbb{R}^{3}}\tilde{u}^2\tilde{v}+\tilde{w}^2\tilde{z}dx,
$$
\begin{equation}
\left|\{\tilde{u},\tilde{w}\}^*\right|_2^2=|u|_2^2+|w|_2^2,\quad\left|\{\tilde{v},\tilde{z}\}^*\right|_2^2=|v|_2^2+|z|_2^2.
\nonumber
\end{equation}
Therefore,
\begin{equation*}		
J(\{\tilde{u},\tilde{w}\}^*,\{\tilde{v},\tilde{z}\}^*)<J(\tilde{u},\tilde{v})+J(\tilde{w},\tilde{z})\label{3.20}		
\end{equation*}
and
\begin{equation*}
J(\{\tilde{u},\tilde{w}\}^*,\{\tilde{v},\tilde{z}\}^*)\ge m\left(\left|\{\tilde{u},\tilde{w}\}^*\right|_2^2,\left|\{\tilde{v},\tilde{z}\}^*\right|_2^2\right)
\ge m\left(|u|_2^2+|w|_2^2,|v|_2^2+|z|_2^2\right).\label{3.21}
\end{equation*}
Hence, we get
\begin{align}\label{3-23}
		         m\left(|u|_2^2+|w|_2^2,|v|_2^2+|z|_2^2\right)\leq J(\{\tilde{u},\tilde{w}\}^*,\{\tilde{v},\tilde{z}\}^*)<J(\tilde{u},\tilde{v})+J(\tilde{w},\tilde{z}).
		        \end{align}
By \eqref{3-15}-\eqref{3-17}, we can obtain
\begin{align}\label{3-24}
m(|u|_2^2,|v|_2^2)=&m(a,b)-m\left(a-|u|_2^2,b-|v|_2^2\right)\nonumber\\	
=&m(a,b)-m\left(|w|_2^2,|z|_2^2\right)-m\left(a-|u|_2^2-|w|_2^2,b-|v|_2^2-|z|_2^2\right).
		        \end{align}
Combining \eqref{3-18}-\eqref{3-24} with Lemma \ref{3.1.4}, one has
\begin{align}
m\left(|u|_2^2+|w|_2^2,|v|_2^2+|z|_2^2\right)<&m(|u|_2^2,|v|_2^2)+m(|w|_2^2,|z|_2^2)\nonumber\\
=& m(a,b)-m\left(a-|u|_2^2-|w|_2^2,b-|v|_2^2-|z|_2^2\right)\nonumber\\
\leq&m\left(|u|_2^2+|w|_2^2,|v|_2^2+|z|_2^2\right).\nonumber
\end{align}
This provide a contradiction.

Case 2. If $|u|_2^2=a$, $0<|v|_2^2<b,$ then there exists a unique positive radial solution $w_0$ such that $J(0,w_0)=m(0,b-|v|_2^2)$. By the elliptic regularity theory, we obtain $w_0\in C^2(\mathbb{R}^{3})$. Denote by $\tilde{w}_0$ the classical Schwartz symmetric decreasing rearrangement of $w_0$. Then
\begin{equation*}
\left|\tilde{w}_0\right|_r^r=|w_0|_r^r,\ \forall r\in(1,+\infty),\quad\left|\nabla\tilde{w}_0\right|_2^2\leq|\nabla w_0|_2^2.
\end{equation*}
Take $\tilde{u}$, $\tilde{v}$ as above such that
$J(\tilde{u},\tilde{v})=m(|u|_2^2,|v|_2^2).$
By similar arguments as in Case 1, we infer
\begin{align}		
J(\{\tilde{u},0\}^*,\{\tilde{v},\tilde{w}_0\}^*)<&\frac{1}{2}\int_{\mathbb{R}^{3}}(|\nabla \tilde{u}|^2+|\nabla \tilde{v}|^2+|\nabla \tilde{w}_0|^2)dx-\frac{1}{4}\int_{\mathbb{R}^{3}}\tilde{u}^4dx\nonumber\\
-&\frac{1}{6}\int_{\mathbb{R}^{3}}\left(\tilde{v}^3+\tilde{w}_0^3\right)dx-\frac{1}{2}\beta\int_{\mathbb{R}^{3}}\tilde{v}^2\tilde{w}_0dx\nonumber\\
=&J(\tilde{u},\tilde{v})+J(0,\tilde{w}_0).\nonumber
\end{align}
Then, using \eqref{3-19}, we obtain
\begin{equation*}
m(a,b)
<J(\tilde{u},\tilde{v})+J(0,\tilde{w}_0)
=m(a,|v|_2^2)+m(0,b-|v|_2^2)
= m(a,b),
 \end{equation*}
which gives a contradiction.

Case 3. If $0\le|u|_2^2<a, |v|_2^2=b,$ then it is easily seen that $v_n\to v$ in $L^2(\mathbb{R}^3)$. By the Gagliardo-Nirenberg inequality and boundedness of $\{v_n\}$ in $H^1(\mathbb{R}^3)$ we get $v_n\to v$ in $L^p(\mathbb{R}^3)$ for $p\in [2, 6)$. In view of the second equation of system \eqref{1.1}, by the H\"older inequality we deduce that
\begin{eqnarray*}
&&\int_{\mathbb{R}^3}|\nabla (v_n-v)|^2dx+\lambda_2\int_{\mathbb{R}^3}|v_n-v|^2dx\\
&&=\frac{1}{2}\int_{\mathbb{R}^3}(v_n^2-v^2)(v_n-v)dx+\frac{1}{2}\beta \int_{\mathbb{R}^3}(u_n^2-u^2)(v_n-v)dx+o(1)\\
&&\le\frac{1}{2}\int_{\mathbb{R}^3}(v_n^3-v^3)dx+\frac{1}{2}\beta \left(\int_{\mathbb{R}^3}(u_n^2-u^2)^2dx\right)^{\frac{1}{2}}\left(\int_{\mathbb{R}^3}(v_n-v)^2dx\right)^{\frac{1}{2}}+o(1)\\
&&=o(1),
\end{eqnarray*}
which together with Lemma \ref{2.0.13} implies that $v_n\to v$ in $H^1(\mathbb{R}^3)$. Using the second equation again we get
 \begin{eqnarray*}
&&\int_{\mathbb{R}^3}\nabla (v_n-v)\cdot \nabla (u_n-u)dx+\lambda_2\int_{\mathbb{R}^3}(v_n-v)(u_n-u)dx\\
&&=\frac{1}{2}\int_{\mathbb{R}^3}(v_n^2-v^2)(u_n-u)dx+\frac{1}{2}\beta \int_{\mathbb{R}^3}(u_n^2-u^2)(u_n-u)dx+o(1),
\end{eqnarray*}
which implies that $u_n\to u$ in $L^3(\mathbb{R}^3)$. Then using the interpolation inequality we get $u_n\to u$ in $L^p(\mathbb{R}^3)$ for $p\in[3,6)$.
By the first equation of system \eqref{1.1},
\begin{align*}
\int_{\mathbb{R}^3}|\nabla (u_n-u)|^2dx+\lambda_1\int_{\mathbb{R}^3}|u_n-u|^2dx=&\int_{\mathbb{R}^3}(u_n^3-u^3)(u_n-u)dx+\beta \int_{\mathbb{R}^3}(u_nv_n-uv)(u_n-u)dx+o(1)\\
\le&\left(\int_{\mathbb{R}^3}(u_n^3-u^3)^2dx\right)^{\frac{1}{2}}\left(\int_{\mathbb{R}^3}(u_n-u)^2dx\right)^{\frac{1}{2}}\\
+&\beta \int_{\mathbb{R}^3}\left[(u_n-u)v_n+u(v_n-v)\right](u_n-u)dx+o(1)\\
\le&\left(\int_{\mathbb{R}^3}(u_n^3-u^3)^2dx\right)^{\frac{1}{2}}\left(\int_{\mathbb{R}^3}(u_n-u)^2dx\right)^{\frac{1}{2}}\\
+&\beta \left(\int_{\mathbb{R}^3}(u_n-u)^4dx\right)^{\frac{1}{2}}\left(\int_{\mathbb{R}^3}v_n^2dx\right)^{\frac{1}{2}}\\
+&\beta\left(\int_{\mathbb{R}^3}u^4dx\right)^{\frac{1}{4}}\left(\int_{\mathbb{R}^3}(u_n-u)^4dx\right)^{\frac{1}{4}}\left(\int_{\mathbb{R}^3}(v_n-v)^2dx\right)^{\frac{1}{2}}+o(1)\\
=&o(1),
\end{align*}
which implies that $u_n\to u$ in $H^1(\mathbb{R}^3)$ and thus $|u|_2^2=a$. This produces a contradiction.

To sum up, $|u|_2^2=a, |v|_2^2=b$. Hence, we obtain
$$
\left(u_n(x+y_n),v_n(x+y_n)\right) \rightarrow (u,v)~~\mbox{in}~~L^2(\mathbb{R}^3).$$ From the Gagliardo-Nirenberg inequality and $|\phi_n|_2^2\rightarrow 0$, we obtain $|\phi_n|_q^q \rightarrow 0$ for $q\in[2,6)$.
Similarly, we have $|v_n(x+y_n)-v|_r^r\rightarrow0$ for $r\in[2,6)$. Then, it follows from the H\"older inequality that,
\begin{align}
&\int_{\mathbb{R}^{3}}\left(u_n^2(x+y_n)v_n(x+y_n)-u^2v\right)dx\nonumber\\
=&\int_{\mathbb{R}^{3}}\left(u_n^2(x+y_n)v_n(x+y_n)-u_n^2(x+y_n)v+u_n^2(x+y_n)v-u^2v\right)dx\nonumber\\
\leq&|u_n(x+y_n)|_4^2|v_n(x+y_n)-v|_2+|u_n(x+y_n)-u|_3|u_n(x+y_n)+u|_3|v|_3\rightarrow0.\nonumber
\end{align}
Since $(u,v) \in S(a,b)\cap B_{6b}$, one can see
\begin{align}
m(a,b)\leq J(u,v)\leq&\liminf_{n\rightarrow\infty}J(u_n(x+y_n),v_n(x+y_n))\nonumber\\
=&\liminf_{n\rightarrow\infty}J(u_n,v_n)\nonumber\\
=&m(a,b).\nonumber
\end{align}
We then conclude that
$(u_n(x+y_n),v_n(x+y_n))\rightarrow (u,v)$ in $W$.
\end{proof}

\subsection{Ground states}\label{ground}

Firstly, we have the following Nehari-Pohozaev identity.
\begin{lemma}\label{pohozaev}
		 Let $(u,v)\in H^1(\mathbb{R}^{3})\times H^1(\mathbb{R}^{3})$ be a solution of system $\eqref{1.1}$. Then,
 $$
	\begin{gathered}
	P(u,v):=\int_{\mathbb{R}^{3}}(|\nabla u|^2+|\nabla v|^2)dx-\frac{3}{4}\int_{\mathbb{R}^{3}}u^4dx-\frac{1}{4}\int_{\mathbb{R}^{3}}v^3dx-\frac{3}{4}\beta\int_{\mathbb{R}^{3}}u^2vdx=0.
	\end{gathered}
	$$
	\end{lemma}
\begin{proof}
By the regularity arguments, it follows that $(u, v)$ is a classical solution of $\eqref{1.1}$. Taking $x\cdot\nabla u $ and $x\cdot\nabla v$ as test function to system $\eqref{1.1}$ respectively, we have
\begin{align}
&0=\left(-\Delta u +\lambda_1 u-u^3-\beta uv\right)x\cdot\nabla u,\label{p1}\\
&0=\left(-\Delta v +\lambda_2 v-\frac{1}{2}v^2-\frac{1}{2}\beta u^2\right)x\cdot\nabla v.  \label{p2}
\end{align}
By $\eqref{p1}$, we get
\begin{equation}\begin{aligned}
\lambda_1ux\cdot\nabla u&=div(\frac{\lambda_1}{2}x\cdot u^2)-\frac{3\lambda_1}{2}u^2,\nonumber\\
-\Delta ux\cdot\nabla u&=div(-\nabla ux\cdot\nabla u+x\cdot\frac{|\nabla u|^2}{2})-\frac{1}{2}|\nabla u|^2,\nonumber\\
-u^3x\cdot\nabla u&=div(-\frac{1}{4}x\cdot u^4)+\frac{3}{4}u^4,\nonumber\\
-\beta uvx\cdot\nabla u&=div(-\frac{\beta}{2}x\cdot u^2v)+\frac{3\beta}{2}u^2v+\frac{\beta}{2}u^2x\cdot\nabla v.\nonumber
\end{aligned}
	\end{equation}
Then, it follows that
\begin{eqnarray*}
\int_{\mathbb{R}^{3}}\left[\frac{1}{2}|\nabla u|^2+\frac{3\lambda_1}{2}u^2-\frac{3}{4}u^4-\frac{3\beta}{2}u^2v-\frac{\beta}{2}u^2x\cdot\nabla v\right]dx=0.
\end{eqnarray*}
Taking into account \eqref{p2}, we similarly deduce that
\begin{eqnarray*}
\int_{\mathbb{R}^{3}}\left[\frac{1}{2}|\nabla v|^2+\frac{3\lambda_2}{2}v^2-\frac{1}{2}v^3+\frac{\beta}{2}u^2x\cdot\nabla v\right]dx=0.
\end{eqnarray*}
Therefore, we establish the following Pohozaev identity
\begin{equation*}
\frac{1}{2}\int_{\mathbb{R}^{3}}(|\nabla u|^2+|\nabla v|^2)dx+\frac{3\lambda_1}{2}\int_{\mathbb{R}^{3}}u^2dx+\frac{3\lambda_2}{2}\int_{\mathbb{R}^{3}}v^2dx
-\frac{3}{4}\int_{\mathbb{R}^{3}}u^4dx-\frac{1}{2}\int_{\mathbb{R}^{3}}v^3dx-\frac{3\beta}{2}\int_{\mathbb{R}^{3}}u^2vdx=0.\label{p}
\end{equation*}
Taking $u$ and $v$ as test functions for system $\eqref{1.1}$ respectively, we obtain
\begin{eqnarray*}
&\int_{\mathbb{R}^{3}}\left[|\nabla u|^2+\lambda_1u^2-u^4-\beta u^2v\right]dx=0,\label{po}\\
&\int_{\mathbb{R}^{3}}\left[|\nabla v|^2+\lambda_2v^2-\frac{1}{2}v^3-\frac{1}{2}\beta u^2v\right]dx=0.\label{p4}
\end{eqnarray*}
Thus, we deduce
\begin{equation*}
\int_{\mathbb{R}^{3}}(|\nabla u|^2+|\nabla v|^2)dx-\frac{3}{4}\int_{\mathbb{R}^{3}}u^4dx-\frac{1}{4}\int_{\mathbb{R}^{3}}v^3dx-\frac{3}{4}\beta\int_{\mathbb{R}^{3}}u^2vdx=0.
\end{equation*}
This completes the proof.
\end{proof}
\begin{remark}
By a direct calculation
$$
\frac{d}{dt}J\left(t^{\frac{3}{2}}u(tx),t^{\frac{3}{2}}v(tx)\right)=t\int_{\mathbb{R}^{3}}(|\nabla u|^2+|\nabla v|^2)dx-\frac{3}{4}t^{2}\int_{\mathbb{R}^{3}}u^4dx-\frac{1}{4}t^{\frac{1}{2}}\int_{\mathbb{R}^{3}}v^3dx-\frac{3}{4}t^{\frac{1}{2}}\beta\int_{\mathbb{R}^{3}}u^2vdx,
$$
which implies that
$$\frac{d}{dt}J\left(t^{\frac{3}{2}}u(tx),t^{\frac{3}{2}}v(tx)\right)=\frac{1}{t}P\left(t^{\frac{3}{2}}u(tx),t^{\frac{3}{2}}v(tx)\right).
$$
However, we might not be able to employ a proof similar to \cite{JJLV2022} to show that the local minimizer is a normalized ground state solution.
\end{remark}

Now we denote $K_{a,b}=\left\{(u,v)\in S(a,b):dJ|_{S(a,b)}(u,v)=0\right\}$. From the preceding arguments, it's evident that $K_{a,b}\neq \O$ for any $a\ge0$ and $b\in(0,\bar{b})$. Then, we aim to demonstrate that the following minimization problem
$$
\tilde m(a,b):=\inf\limits_{(u,v)\in K_{a,b}}J(u,v),\quad\forall a\geq0,\ b\in(0,\bar{b})
$$
is well defined.

\begin{lemma}\label{qiangzhi}
		\textnormal{For any $a\ge0$ and $b\in(0,\bar{b})$, $\tilde m(a,b)>-\infty.$}
\end{lemma}
\begin{proof}
For any $a\ge0, b\in(0,\bar{b})$ and $(u,v)\in S(a,b)$, combining the fact that $P(u,v)=0$ with the definition of $J(u,v)$, by the H\"older and Gagliardo-Nirenberg inequality, we can obtain
\begin{equation}\begin{aligned}
  E(u,v)=&\frac{1}{6}\int_{\mathbb{R}^{3}}(|\nabla u|^2+|\nabla v|^2)dx-\frac{1}{12}\int_{\mathbb{R}^{3}}v^3dx-\frac{1}{4}\beta\int_{\mathbb{R}^{3}}u^2vdx\\
          \geq&\frac{1}{6}\left(|\nabla u|_{2}^2+|\nabla v|_{2}^2\right)-\frac{1}{12}|v|_{3}^3-\frac{1}{4}\beta|u|_{4}^2|v|_{2}\\
		  \geq&\frac{1}{6}\left(|\nabla u|_{2}^2+|\nabla v|_{2}^2\right)-\frac{1}{12}C_{3,3}^3|\nabla v|_{2}^{\frac{3}{2}}|v|_{2}^{\frac{3}{2}}-\frac{1}{4}\beta C_{3,4}^2|\nabla u|_{2}^{\frac{3}{2}}|u|_{2}^{\frac{1}{2}}|v|_{2}\\
          =&\frac{1}{6}\left(|\nabla u|_{2}^2+|\nabla v|_{2}^2\right)-\frac{1}{12}C_{3,3}^3b^{\frac{3}{4}}|\nabla v|_{2}^{\frac{3}{2}}-\frac{1}{4}\beta C_{3,4}^2a^{\frac{1}{4}}b^{\frac{1}{2}}|\nabla u|_{2}^{\frac{3}{2}}.\label{3.3.1}
\nonumber\end{aligned}
	\end{equation}
Thus, $E(u,v)$ is coercive and bounded from below. In particular,  $\tilde m(a,b)>-\infty$.
\end{proof}

Now, we shall further characterize the behavior of $\tilde m(a,b)$. By choosing $\bar b>0$ to be sufficiently small and noting that $\tilde m(a,b)\le m(a,b)<0$, we can deduce that $(u,v)\in B_{6b}$ if $\tilde m(a,b)$ is attained at $(u,v)\in S(a,b)$. This implies that $\tilde m(a,b)=m(a,b)$. Additionally, we can establish the following result:

\begin{lemma}\label{3.2.2}
\begin{description}
\item[(i)] $\tilde m(a,b)$ is continuous for $a\ge0$ and $b\in(0,\bar{b})$.
\item[(ii)] If $0\leq c\leq a, 0<d<b<\bar{b}$, then $\tilde m(a,b) \leq \tilde m(c,d)+\tilde m(a-c,b-d)$.
\end{description}
	\end{lemma}

{\it {Completion of proof of Theorem \ref{thm: main ex}.}}
Let $\{(u_n,v_n)\}\subset K_{a,b}$ be a minimizing sequence of $J$ at $\tilde m(a,b)$. Clearly, $\{(u_n,v_n)\}$ is bounded in $W$. Using the fact that $\tilde m(a,b)<0$, it follows that,
there exists $\delta_1>0$ and a sequence $\{y_n\}\subset \mathbb{R}^3$ such that
\begin{equation*}
\limsup_{n\rightarrow\infty}\int_{B_2(y_n)}|(u_n,v_n)|^2dx\geq~\delta_1>0.
\end{equation*}
Following similar arguments as above, we infer that there exists a sequence $\{y_n\}\subset \mathbb{R}^{3}$ such that $$\left(u_n(x+y_n),v_n(x+y_n)\right)\rightarrow (u,v)\in S(a,b)\cap W,
$$
 which implies that
 there exists a minimizer $(u,v)$ for $J$ on $K_{a,b}$. Hence, in view of Lemma \ref{2.0.13}, there exist $\lambda_1,\lambda_2>0$ such that $(u,v)$ is a normalized ground state solution of \eqref{1.1}. The proof is complete.
\\
\\
\\
\\
\\

\noindent \textbf{Ethics approval and consent to participate} Not applicable.
\\

\noindent \textbf{Consent for publication} Written informed consent for publication was obtained from all participants.
\\

\noindent \textbf{Availability of data and material} Not applicable.
\\

\noindent \textbf{Competing interests} The authors declare that they have no conflict of interest to this work.
\\

\noindent \textbf{Funding} This work is supported by National Natural Science Foundation of China (No.12471102, 11971095) and the Beijing Natural Science Foundation (No. 1242007).
\\

\noindent \textbf{Authors' contributions} These authors have contributed equally to the work.
\\

\noindent \textbf{Acknowledgements} The authors are very grateful to the referees for their detailed comments and valuable suggestions, which greatly improved the manuscript.
\\

%\begin{funding}
%This research was supported by NSFC(11971095).
%\end{funding}

%\noindent{\bf Conflict of interest:\rm} The authors declare that there is no conflict of interest.

\noindent\\Qian Gao
\\School of Mathematics and Statistics
\\Northeast Normal University
\\ Changchun 130024
\\China
\\e-mail:gaoq999@nenu.edu.cn
\\
\noindent\\Qun Wang
\\School of Mathematics and Statistics
\\Northeast Normal University
\\ Changchun 130024
\\China
\\e-mail:wq19980228@163.com
\\
\noindent\\Xiaojun Chang
\\School of Mathematics and Statistics \& Center for Mathematics and Interdisciplinary Sciences
\\Northeast Normal University
\\ Changchun 130024
\\China
\\e-mail:changxj100@nenu.edu.cn


\begin{thebibliography}{99}
%Please refer to the journal's website for the corresponding reference style.

\bibitem{Alves2022}
Alves, C. O., Ji, C., Miyagaki, O. H.: Normalized solutions for a Schr\"{o}dinger equation with critical growth in $\mathbb{R}^{N}.$ Calc. Var. Partial Differ. Equ. {\bf 61}, Paper No. 18, 24 pp. (2022)

\bibitem{A2003}
Albert, J.,  Angulo, P. J.:
Existence and stability of ground-state solutions of a Schr\"odinger-KdV system.
		Proc. Roy. Soc. Edinburgh Sect. A \textbf{133}, 987-1029 (2003)

	

\bibitem{fujia5}
Ardila, A. H.:
Existence and stability of a two-parameter family of solitary waves for a logarithmic NLS-KdV system.
 Nonlinear Anal. \textbf{189}, Paper No. 111563, 23 pp. (2019)	

\bibitem{BdeV2013}
	 Bartsch, T., de Valeriola, S.:
 Normalized solutions of nonlinear Schr\"{o}dinger equations.
Arch. Math. (Basel) \textbf{100}, 75-83 (2013)


\bibitem{BJ2018}
Bartsch, T.,  Jeanjean, L.:
Normalized solutions for nonlinear Schr\"odinger systems.
Proc. Roy. Soc. Edinburgh Sect. A \textbf{148}, 225-242 (2018)

\bibitem{BJS2016}
 Bartsch, T., Jeanjean, L., Soave N.:
Normalized solutions for a system of coupled cubic Schr\"{o}dinger equations on $R^3$.
  J. Math. Pures Appl. \textbf{106}, 583-614 (2016)

\bibitem{B2017}
Bartsch, T., Soave, N.:
 A natural constraint approach to normalized solutions of nonlinear Schr\"{o}dinger equations and systems.
   J. Funct. Anal. \textbf{272}, 4998-5037 (2017)
       			
%\bibitem{fujia0}
 %      T. Bartsch, N. Soave;
  %      Multiple normalized solutions for a competing system of Schr\"odinger equations,
   %     \emph{Calc. Var. Partial Differential Equations}, \textbf{58} (2019), Paper No. 22, 24 pp.	
		
  \bibitem{BZ2021}
Bartsch, T., Zhong, X. X., Zou, W. M.:
Normalized solutions for a coupled Schr\"odinger system.
		Math. Ann. \textbf{380}, 1713-1740 (2021)

\bibitem{B2018}
Bhattarai, S., Corcho, A., Panthee, M.:
 Well-posedness for multicomponent Schr\"odinger-gKdV systems and stability of solitary waves with prescribed mass.
J. Dynam. Differential Equations \textbf{30}, 845-881 (2018)

\bibitem{B2021}
Bieganowski, B., Mederski, J.:
Normalized ground states of the nonlinear Schr\"odinger equation with at least mass critical growth.
J. Funct. Anal. \textbf{280}, Paper No. 108989, 26 pp. (2021)

   \bibitem{LT1982}
		Cazenave, T., Lions, P. L.:
 Orbital stability of standing waves for some nonlinear Schr\"{o}dinger equations.
Comm. Math. Phys. \textbf{85}, 549-561 (1982)

  \bibitem{CLY2023}
Chang, X. J., Liu, M. T., Yan, D. K.:
Normalized ground state solutions of nonlinear Schr\"odinger equations involving exponential critical growth.
J. Geom. Anal. \textbf{33}, Paper No. 83, 20 pp. (2023)

 \bibitem{E2015}
Colorado, E.:
Existence of bound and ground states for a system of coupled nonlinear Schr\"odinger-KdV equations.
C. R. Math. Acad. Sci. Paris \textbf{353}, 511-516 (2015)

\bibitem{C2017}
Colorado, E.:
On the existence of bound and ground states for some coupled nonlinear Schr\"odinger-Korteweg-de Vries equations.
Adv. Nonlinear Anal. \textbf{6}, 407-426 (2017)

\bibitem{E2023}
Colorado, E., L\'opez-Soriano, R., Ortega, A.:
Bound and ground states of coupled ``NLS-KdV" equations with Hardy potential and critical power.
J. Differential Equations \textbf{365}, 560-590 (2023)

\bibitem{D2010}
Dias, J., Figueira, M., Oliveira, M.:
Existence of bound states for the coupled Schr\"odinger-KdV system with cubic nonlinearity.
C. R. Math. Acad. Sci. Paris \textbf{348}, 1079-1082 (2010)

%\bibitem{13}
 %  L. C. Evans;
  %  Partial Differential Equations,
	%\emph{American Mathematical Society}, (1998).

\bibitem{G2016}
		Gou, T. X., Jeanjean, L.:
Existence and orbital stability of standing waves for nonlinear Schr\"odinger systems.
Nonlinear Anal. \textbf{144}, 10-22 (2016)

\bibitem{GJ2018}
Gou, T. X., Jeanjean, L.:
 Multiple positive normalized solutions for nonlinear Schr\"odinger systems.
Nonlinearity \textbf{31}, 2319-2345 (2018)

\bibitem{G2019}
     Guo, Y. J., Li, S., Wei, J. C., Zeng, X. Y.:
         Ground states of two-component attractive Bose-Einstein condensates I: Existence and uniqueness.
 J. Funct. Anal. \textbf{276},  183-230 (2019)


\bibitem{GS2014}
Guo, Y. J., Seiringer, R.: On the mass concentration for Bose-Einstein condensates with attractive interactions.
 Lett. Math. Phys. \textbf{104}, 141-156 (2014)


\bibitem{H2019}
Hirata, J., Tanaka, K.:
         Nonlinear scalar field equations with $L^2$ constraint: mountain pass and symmetric mountain pass approaches.
         Adv. Nonlinear Stud. \textbf{19}, 263-290 (2019)

\bibitem{h4}
 Ikoma, N.:
 Compactness of minimizing sequences in nonlinear Schr\"odinger systems under multiconstraint conditions.
Adv. Nonlinear Stud. \textbf{14}, 115-136 (2014)
		
\bibitem{IT2019}
	 Ikoma, N., Tanaka, K.:
 A note on deformation argument for $L^2$ normalized solutions of nonlinear Schr\"odinger equations and systems.
 Adv. Differential Equations \textbf{24}, 609-646 (2019)

\bibitem{J1997}
	Jeanjean, L.:
 Existence of solutions with prescribed norm for semilinear elliptic equations.
Nonlinear Anal. \textbf{28}, 1633-1659 (1997)

\bibitem{JJLV2022}
Jeanjean, L., Jendrej, J., Le, T. T., Visciglia, N.:
Orbital stability of ground states for a Sobolev critical Schr\"odinger equation.
J. Math. Pures Appl. \textbf{164}, 158-179 (2022)



 \bibitem{JL2022}
 Jeanjean, L., Le, T. T.:
         Multiple normalized solutions for a Sobolev critical Schr\"odinger equation.
         Math. Ann.  \textbf{384}, 101-134 (2022)


   \bibitem{J2023}
Jeanjean, L., Zhang, J. J., Zhong, X. X.:
Normalized ground states for a coupled Schr\"odinger system: Mass super-critical case.
 arXiv: 2311.10994 (2023)

\bibitem{K1895}
Korteweg, D. J., de Vries, G.:
 On the change of form of long waves advancing in a rectangular canal, and on a new type of long stationary waves.
Philos. Mag. \textbf{39}, 422-443 (1895)


 \bibitem{LWT2022}
Liang, F. F., Wu, X. P., Tang, C. L.:
Normalized ground-state solution for the Schr\"odinger-KdV system.
Mediterr. J. Math. \textbf{19}, Paper No. 254, 15 pp. (2022)

\bibitem{fujia6}
Liao, F., Zhang, L. M.:
      High accuracy split-step finite difference method for Schr\"odinger-KdV equations.
 Commun. Theor. Phys. \textbf{70}, 413-422 (2018)

\bibitem{duichen}
 Lieb, E. H., Loss, M.:
    Analysis. American Mathematical Society (2001)

%\bibitem{lions}
%J. L. Lions;
 %Quelques méthodes de résolution des problèmes aux limites non linéaires,
  %\emph{(French) Dunod Gauthier-Villars,} (1969).

    \bibitem{L2013}
Liu, C. G., Zheng, Y. Q.:
        On soliton solutions to a class of Schr\"odinger-KdV systems.
        Proc. Amer. Math. Soc. \textbf{141}, 3477-3484 (2013)

\bibitem{LWY2022}
 Luo, X., Wei, J. C., Yang, X. L., Zhen, M. D.:
Normalized solutions for Schr\"odinger system with quadratic and cubic interactions.
J. Differential Equations, \textbf{314}, 56-127 (2022)

\bibitem{M2022}
Mederski, J., Schino, J.:
    Least energy solutions to a cooperative system of Schr\"odinger equations with prescribed $L^2$-bounds: at least $L^2$-critical growth.
  Calc. Var. Partial Differential Equations \textbf{61}, Paper No. 10, 31 pp. (2022)

\bibitem{N2019}
    Noris, B., Tavares, H., Verzini, G.:
     Stable solitary waves with prescribed $L^2$-mass for the cubic Schr\"odinger system with trapping potentials.
   Discrete Contin. Dyn. Syst. \textbf{35}, 6085-6112 (2015)

\bibitem{M2014}
		Shibata, M.:
Stable standing waves of nonlinear Schr\"odinger equations with a general nonlinear term.
  Manuscripta Math. \textbf{143}, 221-237 (2014)

\bibitem{Shibata-2017}
		Shibata, M.:
A new rearrangement inequality and its application for $L^2$-constraint minimizing problems.
 Math. Z. \textbf{287}, 341-359 (2017)


\bibitem{NS2020}
Soave, N.:
 Normalized ground states for the NLS equation with combined nonlinearities.
  J. Differential Equations \textbf{269}, 6941-6987 (2020)

 \bibitem{S2020}
	Soave, N.:
 Normalized ground states for the NLS equation with combined nonlinearities: the Sobolev critical case.
J. Funct. Anal. \textbf{279}, Paper No. 108610, 43 pp. (2020)

%\bibitem{S1983}
%         C. A. Stuart;
 %         Bifurcation from the essential spectrum,
  %        \emph{Lecture Notes in Math.}, \textbf{1017} (1983), 575-596.

\bibitem{W2022}
Wei, J. C., Wu, Y. Z.:
        Normalized solutions for Schr\"odinger equations with critical Sobolev exponent and mixed nonlinearities.
 J. Funct. Anal. \textbf{283}, Paper No. 109574, 46 pp. (2022)

\bibitem{Weinstein}
  Weinstein, M. I.:
     Nonlinear Schr\"odinger equations and sharp interpolation estimates.
 Comm. Math. Phys. \textbf{87}, 567-576 (1983)
 		

\bibitem{W1996}
Willem, M.: Minimax Theorems. Birkh\"auser Verlag, Boston (1996)

\end{thebibliography}
\end{document}